\documentclass[12pt]{amsart}

\usepackage[utf8]{inputenc}
\usepackage[english]{babel}
\usepackage{amsmath}
\usepackage{amsfonts}
\usepackage{amssymb}
\usepackage{graphicx}
\usepackage{mathtools}
\usepackage{amsthm}

\newtheorem{thm}{Theorem}[section]
\newtheorem{prop}[thm]{Proposition}
\newtheorem{lem}[thm]{Lemma}
\newtheorem{cor}[thm]{Corollary}
\newtheorem{rem}[thm]{Remark}

\theoremstyle{definition}
\newtheorem{defi}{Definition}[section]
\newcommand{\R}{\mathbb{R}}
\renewcommand{\P}{\mathbb{P}}
\newcommand{\N}{\mathbb{N}}

\newcommand{\C}{\mathbb{C}}

\newcommand{\E}{\mathbb{E}}
\newcommand{\Conf}{\mathrm{Conf}}

\newcommand{\K}{\mathcal{K}}

\newcommand{\cP}{\mathcal{P}}
\newcommand{\cK}{\mathcal{K}}
\newcommand{\La}{\Lambda}
\newcommand{\la}{\lambda}
\newcommand{\cM}{\mathcal{M}}
\newcommand{\cN}{\mathcal{N}}
\newcommand{\cS}{\mathcal{S}}
\newcommand{\cW}{\mathcal{W}}
\newcommand{\trivial}{\mathbf{1}}
\newcommand{\cF}{\mathcal{F}}
\newcommand{\bra}{\langle}
\newcommand{\ket}{\rangle}
\newcommand{\var}{\mathrm{var}}

\newcommand{\tr}{\mathrm{Tr}}
\newcommand{\locally}{\mathcal{I}_{1,\mathrm{loc}}}
\newcommand{\ran}{\mathrm{ran}}
\newcommand{\cH}{\mathcal{H}}
\newcommand{\supp}{\mathrm{supp}}

\usepackage{fullpage}

\usepackage{hyperref}

\title{Accumulated spectrograms for hyperuniform determinantal point processes}
\author[M. Katori]{Makoto Katori}
\address{Department of Physics, Faculty of Science and Engineering, Chuo University, Kasuga,
Bunkyo-ku, Tokyo 112-8551, Japan}
\email{makoto.katori.mathphys@gmail.com}  
\author[P. Lazag]{Pierre Lazag}
\address{SISSA, via Bonomea 265, 34136, Trieste, Italy\\
 and LAREMA, UMR CNRS 6093, 2 Boulevard Lavoisier
49045 Angers cedex 01, France}
\email{pierrelazag@hotmail.fr}  
\date{}
\author[T. Shirai]{Tomoyuki Shirai}
\address{Institute of Mathematics for Industry
Kyushu University
744 Motooka, Nishi-ku,
Fukuoka 819-0395, Japan}
\email{shirai@imi.kyushu-u.ac.jp}  
\begin{document}

\maketitle

\begin{abstract}
    We define the accumulated spectrogram associated to a locally trace class orthogonal projection operator and  to a bounded set using the polar decomposition of its restriction on that set. We prove a convergence theorem for accumulated spectrograms along an exhaustion. We show that a radial determinantal point process on $\R^d$ is always hyperuniform along the exhaustion formed by the dilations of a bounded open set, and as a consequence, we obtain that dilations of the corresponding accumulated spectrogram converge to the indicator function of the considered set, establishing thus a universal phenomenon. Our result is a generalisation of a theorem by Abreu-Gr\"ochenig-Romero in \cite{AGR16} concerning time-frequency localization operators.
\end{abstract}
\section{Introduction}

In \cite{AGR16}  the authors associate to time-frequency localization operators and a bounded set $\Lambda \subset \R^{2d}$ a function in $L^1(\R^{2d})$, which they call \emph{accumulated spectrogram}, and prove a convergence theorem for dilations of this function to the indicator function of the set $\La$. Our aim is to generalize this construction together with the convergence theorem to any locally trace class orthogonal projection giving rise to a hyperuniform determinantal point process on some space state $S$ equipped with a Radon measure $\lambda$.\\

We construct the equivalent of the accumulated spectrogram associated to an orthogonal projection  and a bounded set $\La \subset S$ using the polar decomposition of that operator resticted to $\La$ on the right in Section \ref{sec:polardecompo} below, see Definition \ref{def:accspec}.\\

We formulate our main results in Section \ref{sec:mainresult}. We first state the abstract convergence theorem for accumulated spectrograms, Theorem \ref{thm:L1-limit} below. In particular, in the case when the state space is the euclidean space $\R^d$ equipped with the Lebesgue measure and when the kernel of the orthogonal projection is constant on its diagonal, the dilations along a fixed bounded set of the accumulated spectrogram converge to the indicator function of that set, see Corollary \ref{cor:L1-lim1}. Since the limit we obtain does not depend on the specific orthogonal projection, we face here a universality phenomenon.

The assumption of hyperuniformity we need can be expressed analytically, without linking it with the concept of determinantal point processes, see Remark \ref{rem:hyperunif}. However, it has a simple and comprehensive probabilistic interpretation once the latter notion has been introduced. For further background on hyperuniform point processes, see the surveys \cite{Torquato2018} and \cite{GhoshLebowitz2017}.

We prove that any radial determinantal point process is hyperuniform along any exhaustion by a dilation of a bounded open set, see Proposition \ref{prop:hyperunifradial} below. Consequently, the assumptions of Theorem \ref{thm:L1-limit} and Corollary \ref{cor:L1-lim1} are satisfied in that case, which leads to the result stated in Corollary \ref{cor-L1lim2}. We thus recover the main result from \cite{AGR16}, where the orthogonal projection is constructed out from time-frequency localization operators, as a particular example.

Other examples include the \emph{Heisenberg family of determinantal point processes} (see \cite{MKS21}, \cite{K22}, \cite{KS22a}), also sometimes called Heisenberg or Weyl-Heisenberg ensembles (see \cite{APRT17}, \cite{AGR19}) as well as higher dimensional generalizations of the sine process, namely the \emph{Euclidean family of determinantal point process} (\cite{KS22a}), and is also called \emph{Fermi-sphere ensembles} in \cite{TSZ2008}, \cite{SZT2009}, \cite{Torquato2018}. We perform explicit computations for the latter at the end of this article by means of the geometric method we use to prove hyperuniformity for all radial determinantal point processes. The Euclidean family of determinantal point process also appears as local semi-classical limits of fermions point fields on compact Riemannian manifolds (\cite{KS22b}) or on $\R^d$ for a wide class of potentials (\cite{DeleporteLambert}). Both types, Heisenberg and Euclidean family of determinantal point processes, appeared in \cite{Zel00} (see also \cite{Zel09} and references therein) where the terminology is taken from, in the study of asymptotics of zeroes of random polynomials or random waves, respectively for complex or real Riemannian manifolds.

\subsection{Polar decomposition and accumulated spectrogram associated to a locally trace class orthogonal projection} \label{sec:polardecompo}
The material we present here on polar decomposition is very standard, and can be found for instance in \cite{Reed-Simon1}, ch. 6. Our presentation is however quite detailed, to the extent that we will subsequently need elements of the construction. For the notion and properties of trace class operators, see e.g. \cite{Sim05}.

Let $S$ be a locally compact, complete and separable metric space, and let $\lambda$ be a Radon measure on $S$. We consider the separable complex Hilbert space $L^2(S,\lambda)$ with inner product $\bra, \ket$ being linear in the first variable
\[ \bra f , g \ket := \int_S f(x) \overline{g(x)} d\lambda(x), \quad f,g \in L^2(S, \lambda).\]
Let $\Lambda \subset S$ be a bounded Borel set and denote by $\cP_\Lambda$ the restriction operator
\begin{align*}
\cP_\Lambda : L^2(S,\lambda) &\to L^2(S,\lambda) \\
f & \mapsto \trivial_\Lambda f,
\end{align*}
where $\trivial_\Lambda$ is the indicator function of $\Lambda$. We let $\locally(S,\lambda)$ be the ideal of locally trace class operators on $L^2(S,\lambda)$, i.e. the ideal of operators $\cK : L^2(S,\lambda) \rightarrow L^2(S,\lambda)$ such that for any bounded set $\La \subset S$, the operator $\cP_\La \cK \cP_\La$ is of trace class. A locally trace class operator $\cK$ admits a kernel, i.e. a function $K : S \times S \rightarrow \C$ such that for all $f \in L^2(S,\lambda)$ and $\lambda$-almost every $x \in S$
\[\cK f(x) = \int_S K(x,y)f(y) d\lambda(y).
\] 
Let $\cK \in \locally(S,\lambda)$ be an orthogonal projection onto a closed subspace $L$ of $L^2(S,\lambda)$. Since $L^2(S,\lambda)$ is separable, we can consider a Borel set $S_0 \subset S$ verifying $\lambda(S \setminus S_0)=0$ and such that every function $f \in L$ is defined everywhere on $S_0$. Such an operator defines a determinantal point process $\P$, see Section \ref{sec:dpp} below. For a bounded set $\La \subset S$, the operator 
\begin{align} \label{eq:defMlambda}
\cM_\La =\cP_\La \cK \cP_\La
\end{align}
is of trace class and admits the spectral decomposition 
\begin{align} \label{eq:specdecompo1}
\cM_\Lambda = \sum_{j=1}^{+\infty} \mu_j^{(\Lambda)} \Phi_j^{(\Lambda)} \otimes \Phi_j^{(\Lambda)},
\end{align}
where $\{\mu_j^{(\Lambda)} \}_{j\geq 1} \subset [0,1]$ is the set of eigenvalues of $\cM_\Lambda$ ordered in the non-increasing order 
\begin{align*}
1 \geq \mu_1^{(\Lambda)} \geq \mu_2^{(\Lambda)} \geq ... \geq 0,
\end{align*}
and $\Phi_j^{(\Lambda)} \in L^2(S,\lambda)$ is the normalized eigenfunction of $\cM_\Lambda$ corresponding to the eigenvalue $\mu_j^{(\Lambda)}$. The functions $\Phi_j^{(\La)}$ form an orthonormal basis of $L^2(S,\lambda)$. The representation (\ref{eq:specdecompo1}) means that, for any $f \in L^2(S,\lambda)$, one has the decomposition
\begin{align*}
\cM_{\La}f= \sum_{j=1}^{+\infty} \mu_j^{(\La)} \bra f , \Phi_j^{(\La)} \ket \Phi_j^{(\La)}.
\end{align*}
\begin{rem}
    The representation (\ref{eq:specdecompo1}) is essentially unique, in the sense that lack of uniqueness only arises for eigenvalues of multiplicity greater than one.
\end{rem}
Consider now the operator
\begin{align} \label{eq:defNlambda}
\cN_\La = \cK \cP_\La.
\end{align}
Its adjoint operator is the operator $\cP_\La \cK$ so that $\cN_\La^* \cN_\La = \cM_\La$, and we have the polar decomposition
\begin{align} \label{eq:polardecompo0}
\cN_\La = \cW_\La \sqrt{\cM_\La},
\end{align}
where $\cW_\La$ is the operator defined as follows: for $g=\sqrt{\cM_\La} f$, we set $\cW_\La g = \cN_\La f$, and we set $\cW_\La g = 0 $ for $g \in \ker(\cM_\La) = \ker( \cN_\La )$. Since $(\cW_\La)_{|\ran (\cW_\La)}$ is an isometry and since we have the orthogonal decomposition $L^2(S, \lambda) = \overline{\mathrm{ran}( \cM_\La ^* )} \oplus \mathrm{ker}(\cM_\La)$, the operator $\cW_\La$ is well defined. For $j \in \N$, we define
\begin{align}
\Psi_j^{(\La)} = \cW_\La \Phi_j^{(\La)},
\end{align}
so that we can write the polar decomposition (\ref{eq:polardecompo0}) as
\begin{align}  \label{eq:polardecompo}
\cN_\La = \sum_{j=1}^{+\infty} \sqrt{\mu_j^{(\La)}} \Phi_j^{(\La)} \otimes \Psi_j^{(\La)},
\end{align}
meaning that 
\begin{align*}
\cN_\La f= \sum_{j=1}^{+ \infty} \sqrt{\mu_j^{(\La)}} \bra f , \Phi_j^{(\La)} \ket \Psi_j^{(\La)}, \quad f \in L^2(S,\lambda).
\end{align*}
Observe that by construction, we have for all $j \in \N$ that
\begin{align} \label{eq:exprpsi}
    \cN_\La \Phi_j^{(\La)} = \sqrt{\mu_j^{(\La)}} \Psi_j^{(\La)}.
\end{align}
We associate to the operator $\cK$ and the set $\La \subset S$ a function in $L^1(S,\lambda)$ called the \emph{accumulated spectrogram} in the following way.
\begin{defi} \label{def:accspec}
The {\it accumulated spectrogram} associated to $(\cK, \Lambda)$ is the function $\rho_\Lambda \in L^1(S,\lambda)$ defined by 
\begin{align*}
\rho_\Lambda (x):= \sum_{j=1}^{N_\Lambda} |\Psi_j^{(\Lambda)}(x)|^2, \quad x \in S_0,
\end{align*}
where $N_\Lambda$ is the upper integer part of the trace of $\cM_\La$, i.e. the smallest integer that is larger than $\tr(\cM_\La)$.
\end{defi}
\subsection{Main result} \label{sec:mainresult}
For $f \in L^1(S,\lambda)$, we denote by $||f||_1$ its $L^1$ norm. We are interested in accumulated spectrograms along an exhaustion of $S$.
\begin{defi}An {\it exhaustion} $\{\Lambda^{(n)} \}_{n \in \N}$ is a family of subsets of $S$ such that 
\begin{enumerate}
 \item[$\bullet$] each $\Lambda^{(n)}$ is a Borel relatively compact set;
 \item[$\bullet$] for all $n \in \N$, one has the inclusion $\Lambda^{(n)} \subset \Lambda^{(n+1)}$;
 \item[$\bullet$] the sets $\La^{(n)}$ cover $S$: $\cup_{n \in \N} \Lambda^{(n)} = S$.
\end{enumerate} 
\end{defi}
Our main result is the following. We refer to Section \ref{sec:dpp} below for the notion of hyperuniform determinantal point processes. 
\begin{thm} \label{thm:L1-limit} Fix an exhaustion $\{ \Lambda^{(n)} \}_{n\in \N}$ of $S$ and assume that the determinantal point process associated to $\cK$ is hyperuniform along this exhaustion. Then one has the convergence in $L^1(S,\lambda)$ 
\begin{align} \label{lim:thm}
\lim_{n \rightarrow + \infty} \frac{1}{N_{\Lambda^{(n)}}} || \rho_{\Lambda^{(n)}} - K(\cdot,\cdot)\trivial_{\Lambda^{(n)}} ||_1 = 0.
\end{align}
\end{thm}
In the case where the state space $S$ is the euclidean space $\R^d$ equipped with the Lebesgue measure $dx$ and when the kernel $K$ of $\K$ is constant on its diagonal, Theorem \ref{thm:L1-limit} implies that dilations of the accumulated spectrograms form an approximation of the indicator function of a relatively compact set in $\R^d$.
\begin{cor} \label{cor:L1-lim1}Let $d\geq 1$ be a positive integer and consider the case $(S,\lambda)=(\R^d,dx)$. Let $\Lambda\subset \R^d$ be a bounded Borel set, let $K$ be the kernel of a locally trace class orthogonal projection  $\K$ on $L^2(\R^d,dx)$ such that the corresponding determinantal point process is hyperuniform along the exhaustion $\{R\Lambda \}_{R>0}$. Assume that $K$ is constant on the diagonal and normalized, i.e. that we have $K(x,x)=1$ for alomst all $x \in \R^d$. We then have the convergence in $L^1(\R^d,dx)$:
\begin{align*}
\rho_{R\Lambda}(R \cdot) \overset{L^1}{\longrightarrow} \trivial_\Lambda
\end{align*}
as $R \to + \infty$.
\end{cor}
\subsection{The case of radial determinantal point processes on $\R^d$}
A particularly interesting case is the one of radial determinantal point processes  on $\R^d$ coming from orthogonal projection kernels. We define a radial determinantal point process as a determinantal point process having a radially symmetric correlation kernel.
\begin{defi} \label{def:radial}
Let $\cK$ be a locally trace class orthogonal projection defined on $L^2(\R^d,dx)$ with kernel $K$ and let $\P$ be the corresponding determinantal point process. We say that the point process $\P$ is \emph{radial} if there exists a function
\[\varphi : [0,+\infty) \to [0,+ \infty) \]
such that for alomst all $x,y \in \R^d$, we have
\[|K(x,y)|^2 = \varphi( ||x-y||), \]
where $||x-y||$ is the Euclidean distance from $x$ to $y$ in $\R^d$.
\end{defi}
In particular, the correlation kernel of a radial determinantal point process is constant on its diagonal, and we assume that it is normalized such that $K(x,x) \equiv 1$.
\begin{cor} \label{cor-L1lim2}
Let $\P$ be a radial determinantal point process on $\R^d$ coming from a locally trace class orthogonal projection. Then, for any relatively compact open set $\La \subset \R^d$, we have the convergence in $L^1(\R^d,dx)$:
\begin{align*}
\rho_{R\Lambda}(R \cdot) \overset{L^1}{\longrightarrow} \trivial_\Lambda
\end{align*}
as $R \to + \infty$.
\end{cor}
Corollary \ref{cor-L1lim2} is a direct consequence of Corollary \ref{cor:L1-lim1} and the following Proposition asserting that any radial determinantal point process is hyperuniform along any exhaustion of the form $\{R \La \}_{R>0}$.
\begin{prop} \label{prop:hyperunifradial} 
Let $\P$ be a radial determinantal point process on $(\R^d,dx)$ governed by an orthogonal projection operator. Then $\P$ is hyperuniform along any exhaustion $\{ R\La \}_{R >0}$, where $\La \subset \R^d$ is a bounded open set with positive Lebesgue measure.
\end{prop}

\subsection{Organization of the paper}
We recall the definition of hyperuniform point processes in section \ref{sec:dpp} as well as the notion of determinantal point processes and some of their basic properties.

In section \ref{sec:proof}, we prove our main result, Theorem \ref{thm:L1-limit} together with its Corollary \ref{cor:L1-lim1}.

We prove Proposition \ref{prop:hyperunifradial} in section \ref{sec:radial}.

Section \ref{sec:ex} is devoted to the description of some examples.

\subsection*{Acknowledgements} M.S. was supported by JSPS KAKENHI Grant Numbers JP19K03674, JP21H04432, JP22H05105, and JP23H01077. P.L. was supported by the project ULIS 2023-09915 from R\'egion Pays de la Loire. T.S. was supported by JSPS KAKENHI Grant Numbers JP20K20884, JP22H05105 and JP23H01077. T.S. was also supported in part by JSPS KAKENHI Grant Numbers JP20H00119 and JP21H04432.

%\subsection*{Conflict of Interest Statement}The authors have no conflicts to disclose.

%\subsection*{Data Availability Statement}Data sharing is not applicable to this article as no new data were created or analyzed in this study.

\section{Determinantal point processes}\label{sec:dpp}

\subsection{Hyperuniform point processes}

A configuration $\Xi$ on $S$ is a non-negative integers valued Radon measure on $S$. The space of configuration, denoted by $\text{Conf}(S)$, is again a complete separable metric space, and we equip it with its Borel sigma-algebra. A configuration $\Xi \in \text{Conf}(S)$ is said to be {\it simple} if one has $\Xi(\{x\}) \in \{0,1\}$ for any $x \in S$.
\begin{defi}A {\it point process} on $S$ is a probability measure on $\text{Conf}(S)$. A point process is simple if it is supported on simple configurations.
\end{defi}
See \cite{daley-verejones} for a general treatment of point processes. Informally, a point process is said to be hyperuniform if it has a small number variance. Here and in the sequel, $\E$ denotes the expectation with respect to the point process $\P$ and if $X \in L^2(\Conf(S),\P)$, $\var(X)$ denotes the variance of the random variable $X$, i.e.
\[\var(X) = \E\left[ (X - \E [X] )^2 \right]. \] 
\begin{defi}
Let $\P$ be a point process on $S$ such that for any relatively compact Borel set $\La \subset S$, we have $\Xi(\La) \in L^2(\Conf(S), \P)$. The point process $\P$ on $S$ is {\it hyperuniform} along an exhaustion 
$\{\La^{(n)}\}_{n \ge 1}$ of $S$ if 
\[
 \lim_{n \to \infty} \E[\Xi(\La^{(n)})] = \infty, \quad 
 \lim_{n \to \infty}
 \frac{\var(\Xi(\La^{(n)}))}{\E[\Xi(\La^{(n)})]} = 0. 
\]
\end{defi}

\subsection{Determinantal point processes}

Let $\lambda$ be a Radon measure on $S$. For a simple point process $\P$ on $S$ and for a positive integer $n \in \N$, the $n$-th correlation function with respect to $\lambda^{\otimes n}$, denoted by $\rho_n : S^n \rightarrow \C$, if it exists, is defined by
\begin{multline}\label{eq:correl}\int_{\text{Conf}(S)} \overset{*}{\sum_{x_1,...,x_n \in \supp (\Xi)}}f(x_1,...,x_n) d\P(\Xi)
 = \int_{S^n} f(x_1,...,x_n) \rho_n(x_1,...,x_n) d\lambda(x_1)...d\lambda(x_n),
\end{multline}
for any measurable compactly supported function $f : S^n \rightarrow \C$. Here, the sum $\overset{*}{\sum}$ is taken over all pairwise distinct points $x_1,...,x_n \in \mathrm{supp}(\Xi)$.
\begin{defi}\label{def:dpp}A point process $\P$ on $S$ is a {\it determinantal point process} on $(S,\lambda)$ if there exists a kernel 
\[ K : S \times S \rightarrow \C.
\]
such that, for any $n \in \N$, the $n$-th correlation function exists and is given by
\[\rho_n(x_1,...,x_n) = \det \left( K(x_i,x_j) \right)_{i,j=1}^n.
\]
\end{defi}
See \cite{borodindet}, \cite{soshnikov}, \cite{ST03a} and \cite{ST03b} for further background on determinantal point processes. A kernel $K$ defines an integral operator $\cK$
\begin{align*} \cK : \hspace{0.1cm} L^2(S,\lambda) &\rightarrow L^2(S,\lambda) \\
f &\mapsto \left( x \mapsto \int_S f(y)K(x,y) d\lambda(y) \right).
\end{align*}
The following Theorem provides sufficient conditions for the kernel of a locally trace class operator to be the correlation kernel of a determinantal point process.
%For a bounded set $\Lambda \subset S$, we denote by $\cP_\Lambda$ the restriction operator onto $\Lambda$, i.e. the operator of multiplication by the indicator function of $\La$, denoted by $\trivial_\La$. An integral operator $\cK$ with kernel $K$ is locally of trace class if for any bounded borel set $\Lambda \subset S$, the operator $\cP_\Lambda \cK \cP_\Lambda$ is of trace class, i.e. :
%\[  \left| \int_\Lambda K(x,x) d\lambda(x) \right| <+ \infty. \]
%We denote by $\locally(S,\lambda)$ the ideal of locally trace class operators on $L^2(S,\lambda)$. All locally trace class operators admit a kernel. We recall the following fundamental result :
\begin{thm}[\cite{Macchi75} \cite{soshnikov}, \cite{ST03a}] \label{thm:existence} Let $\cK \in \locally(S,\lambda)$ be a hermitian locally trace class operator. Then, its kernel $K$ serves as the correlation kernel of a determinantal point process if and only if $0 \leq \cK \leq I$. In particular, any locally trace class orthogonal projection onto a closed subspace $ L \subset L^2(S,\lambda)$ gives rise to a determinantal point process.
\end{thm}
For a determinantal point process $\P$ with kernel $K$ being the kernel of the integral operator $\cK$ and for a bounded Borel set $\La \subset S$, we have from Definition \ref{def:dpp} and formula (\ref{eq:correl}) that
\begin{equation} \label{eq:expectation}
\E [ \Xi (\Lambda) ] = \int_\Lambda K(x,x)d\lambda(x) =\tr ( \cM_\La )= \sum_{j=1}^{+\infty} \mu_j^{(\La)} = \tr ( \cN_\Lambda), 
\end{equation}
where $\cM_\La$ and $\cN_\La$ were respectively defined by (\ref{eq:defMlambda}) and (\ref{eq:defNlambda}), and where the numbers $\mu_j^{(\La)} \in [0,1]$, $j=1,2,\dots$ are the non-increasingly ordered eigenvalues of $\cM_\La$. The last equality in (\ref{eq:expectation}) follows the fact that the operators $\cM_\La$ and $\cN_\La$ have the same non-zero eigenvalues. We also have
\begin{align} \label{eq:variance0}
    \var (\Xi(\La)) = \tr (\cM_\La -\cM_\La^2).
\end{align}
\begin{rem} \label{rem:hyperunif}
    For a determinantal point process given by a locally trace class operator $\cK$ and if $\{ \La^{(n)} \}_{n \geq 1}$ is an exhaustion, we then have that $\P$ is hyperuniform along this exhaustion if and only if
    \begin{align*}
        \lim_{n \to + \infty} \tr(\cM_{\La(n)}) = \tr(\cK) = + \infty
    \end{align*}
    and
    \begin{align*}
        \lim_{n \to + \infty} \frac{\tr((\cM_{\La^{(n)}})^2) }{\tr(\cM_{\La^{(n)}})} = \lim_{n \to + \infty} \frac{\sum_j (\mu_j^{(\La^{(n)})})^2}{\sum_j \mu_j^{(\La^{(n)})}} =1.
    \end{align*}
\end{rem}

In the case when $\cK$ is an orthogonal projection, the reproducing property of its kernel
\begin{align} \label{eq:repprop}
\int_S K(x,y) K(y,z)d\lambda(y) = K(x,z)
\end{align}
implies
\begin{align}\label{eq:variance}
\var( \Xi(\Lambda) ) = \frac{1}{2} \int_S \int_S | \trivial_\La(x) - \trivial_\La(y)| |K(x,y)|^2 d\lambda(x)d\lambda(y),
\end{align}
which also leads to another useful expression 
\begin{equation}\label{eq:variance2}
\var(\Xi(\Lambda))= \int_\Lambda \int_{S \setminus \Lambda} |K(x,y)|^2 d\lambda(x) d\lambda(y).
\end{equation}
Observe that formula (\ref{eq:variance}) implies a negative correlation property in the following sense: if $\La, \La' \subset S$ are disjoint bounded Borel sets, then we have from the triangular inequality that
\begin{align} \label{ineq:negcor}
\var( \Xi( \La \sqcup \La')) \leq \var( \Xi (\La)) + \var( \Xi(\La')).
\end{align}
Also observe that from formula (\ref{eq:variance}) or (\ref{eq:variance2}) and using the reproducing property (\ref{eq:repprop}), we have the upper bound
\begin{align} \label{ineq:varexp}
\var(\Xi(\La)) \leq \E [\Xi(\La) ].
\end{align}
Inequalities (\ref{ineq:negcor}) and (\ref{ineq:varexp}) will be used in the proof of Proposition \ref{prop:hyperunifradial}.

\section{Proof of Theorem \ref{thm:L1-limit}} \label{sec:proof}
We state preliminary results, Lemmas \ref{lem:Ndelta} and \ref{lem:G} below, and prove Theoreom \ref{thm:L1-limit} in Section \ref{subsec:proof} assuming these results. Section \ref{subsec:proof} ends with the proof of Corollary \ref{cor:L1-lim1} from Theorem \ref{thm:L1-limit}. We prove Lemmas \ref{lem:Ndelta} and \ref{lem:G} in Section \ref{subsec:prooflem}.
\subsection{Preliminary results and proofs of Theorem \ref{thm:L1-limit} and Corollary \ref{cor:L1-lim1}} \label{subsec:proof}
\subsubsection{Preliminary results}
Fix a relatively compact set $\Lambda \subset S$. For $x \in S_0$, we define the function
\begin{align*}
K_x : y \mapsto K(y,x),
\end{align*}
and observe that we have
\begin{align}
\cK f (x)= \bra f, K_x \ket
\end{align}
for all $f \in L^2(S,\lambda)$, $x \in S_0$.

We also set
\begin{align*}
G( x) = K(x,x)\trivial_\Lambda(x) - \bra \cN_\Lambda K_x,K_x \ket, \quad x \in S_0.
\end{align*}
Observing that 
\[ \rho_\La(x) -K(x,x)\trivial_\La(x) = \rho_\La(x)- \bra \cN_\Lambda K_x,K_x\ket - G(x),
\]
we first estimate the left hand side of (\ref{lim:thm}) with 
\begin{align} \label{estim1}
||\rho_\Lambda - K( \cdot, \cdot)\trivial_\La ||_1 \leq ||\rho_\Lambda - \bra \cN_\Lambda K_\cdot,K_\cdot\ket ||_1 + ||G||_1.
\end{align}
Theorem \ref{thm:L1-limit} will now follow from Lemmas \ref{lem:Ndelta} and \ref{lem:G} below, each of them estimating a term on the right hand side of inequality (\ref{estim1}).
\begin{lem}\label{lem:Ndelta}
For $\delta \in (0,1)$, set
\begin{align} \label{eq:defcdelta}
C_\delta = \max \left\{ \frac{1}{\delta}, \frac{1}{1-\delta} \right\}.
\end{align}
Then, for any $\delta \in (0,1)$, we have 
\begin{align}
\|
\rho_{\La} - \bra \cN_{\La} K_{\cdot}, K_{\cdot}\ket\|_{1} 
\le 1 + 2 \delta \E[\Xi(\Lambda)] + 
2(1-\delta) C_{\delta}  \var(\Xi(\La)). 
\end{align}
\end{lem}
The second Lemma is a simple estimate for the $L^1$ norm of $G$.
\begin{lem}\label{lem:G}
We have
\[
 \|G\|_{1} \le 2 \var(\Xi(\La)). 
\] 
\end{lem}
\subsubsection{Proof of Theorem \ref{thm:L1-limit}}
We now prove Theorem \ref{thm:L1-limit}, assuming Lemmas \ref{lem:Ndelta} and \ref{lem:G}. 
\begin{proof}[Proof of Theorem \ref{thm:L1-limit}]
Fix a number $\delta \in (0,1)$ and an exhaustion $\{ \Lambda^{(n)} \}_{n \in \N}$ of $S$ such that the determinantal point process $\P$ is hyperuniform along this exhaustion. By the inequality (\ref{estim1}), Lemma \ref{lem:Ndelta} and Lemma \ref{lem:G}, we have that, for all $n \in \N$
\begin{align*}
||\rho_{\La^{(n)}} - K(\cdot,\cdot)\trivial_{\La^{(n)}} ||_1 \leq 1 + 2\delta\E[\Xi(\La^{(n)})] + 2\{(1-\delta)C_\delta+1\}\var(\Xi(\La^{(n)})).
\end{align*}
Dividing this inequality by $N_{\La^{(n)}}$ and using the fact that the point process $\P$ is hyperuniform along the exhaustion $\{\La^{(n)} \}_{n \in \N}$, we obtain 
\begin{align*}
\limsup_{n \rightarrow + \infty} \frac{1}{N_{\La^{(n)}}} ||\rho_{\La^{(n)}} - K(\cdot,\cdot)\trivial_{\La^{(n)}} ||_1 \leq 2\delta.
\end{align*}
Since $\delta>0$ can be chosen arbitrarly small, the Theorem is proved.
\end{proof}
\subsubsection{Proof of Corollary \ref{cor:L1-lim1}}
For $R>0$, we have 
\begin{align*}
||\rho_{R\Lambda}(R \cdot) - \trivial_\La||_1 &= \int_{\R^d}| \rho_{R\La}(Rx) - \trivial_\La(x) |dx \\
&=\int_{\R^d} | \rho_{R\La}(Rx) - \trivial_{R\La}(Rx) | dx.
\end{align*}
Performing the change of variable $x \mapsto Rx$ leads to 
\begin{align} \label{eq:pfcor1}
||\rho_{R\Lambda}(R \cdot) - \trivial_\La||_1 &= \frac{1}{R^d} \int_{R^d} |\rho_{R\La}(x) - \trivial_{R\La}(x)|dx \\
&= \frac{1}{R^d} || \rho_{R\La} - \trivial_{R\La} ||_1.
\end{align}
Since $\La$ is bounded, there exists $C>0$ such that 
\begin{align} \label{ineq:pfcor1}
\text{Leb}(R\La) \le CR^d
\end{align}
for all $R>0$. The fact that $K(x,x) = 1$ for all $x \in \R^d$ and formula (\ref{eq:expectation}) imply that 
\begin{align*}
\E[ \Xi(R\La)] = \text{Leb}(R\La).
\end{align*}
We thus obtain from (\ref{eq:pfcor1}) and (\ref{ineq:pfcor1}) that
\begin{align*}
||\rho_{R\Lambda}(R \cdot) - \trivial_\La||_1 &\leq \frac{C}{\text{Leb}(R\La)} || \rho_{R\La} - \trivial_{R\La} ||_1 \\
&= \frac{C}{\E[\Xi(R\La)]}|| \rho_{R\La} - \trivial_{R\La} ||_1.
\end{align*}
Since the determinantal point process with kernel $K$ is hyperuniform along the exhaustion $\{R\La\}_{R>0}$, Theorem \ref{thm:L1-limit} implies that
\[\frac{C}{\E[\Xi(R\La)]}|| \rho_{R\La} - \trivial_{R\La} ||_1 \rightarrow 0
\]
as $R \rightarrow + \infty$, which concludes the proof.
\subsection{Proof of Lemmas \ref{lem:Ndelta} and \ref{lem:G}} \label{subsec:prooflem}
\subsubsection{Two expressions for $\bra \cN_\La K_x, K_x \ket$}
We give here two different expressions for the inner product $\bra \cN_\La K_x, K_x \ket$. Proposition \ref{prop:innerprod1} will be used in the proof of Lemma \ref{lem:Ndelta}, while Lemma \ref{lem:G} will be a consequence of Proposition \ref{prop:G}.
\begin{prop}\label{prop:innerprod1}For any $x \in S_0$, we have
\begin{align} \label{eq:eqinnerprod1}
\bra \cN_\La K_x , K_x \ket = \sum_{j=1}^{+ \infty} \mu_j^{(\La)} \left| \Psi_j^{(\La)} (x) \right|^2. 
\end{align}
\end{prop}
\begin{proof} From the polar decomposition (\ref{eq:polardecompo}), we have for all $x \in S_0$ that
\begin{align} \label{eq:decompoKx}
\bra \cN_\La K_x, K_x \ket = \sum_{j=1}^{+ \infty} \sqrt{\mu_j^{(\La)}} \bra K_x, \Phi_j^{(\La)} \ket \bra \Psi_j^{(\La)} , K_x \ket.
\end{align}
Let $j \in \N$. Since $\Psi_j^{(\La)} \in \ran(\cK)$, we have for all $x \in S_0$ that
\begin{align} \label{eq:psiprodKx}
\bra \Psi_j^{(\La)}, K_x \ket = \Psi_j^{(\La)} (x).
\end{align}
Now from (\ref{eq:exprpsi}), and since $\cP_\La \Phi_j^{(\La)} = \Phi_j^{(\La)}$, we have 
\begin{align*}
\cK \Phi_j^{(\La)} = \sqrt{\mu_j^{(\La)}}\Psi_j^{(\La)},
\end{align*}
and thus, we obtain that for all $x \in S_0$ that
\begin{align} \label{eq:phiprodKx}
\bra \Phi_j^{(\La)} , K_x \ket = \sqrt{\mu_j^{(\La)}} \Psi_j^{(\La)} (x).
\end{align}
Inserting equalities (\ref{eq:phiprodKx}) and (\ref{eq:psiprodKx}) into formula (\ref{eq:decompoKx}), we obtain equation (\ref{eq:eqinnerprod1}) and the proof is complete.
\end{proof}
\begin{prop} \label{prop:G}For all $x \in S_0$, we have
\begin{align}
\bra \cN_\La K_x, K_x \ket = \int_\La |K(x,y) | ^2 d\lambda(y).
\end{align}
\end{prop}
\begin{proof}Since the operator $\cK$ is self-adjoint and since $\cK K_x = K_x$ for all $x \in S_0$, we have
\begin{align*}
\bra \cN_\La K_x, K_x \ket = \bra \cK \cP_\La K_x , K_x \ket = \bra\cP_\La K_x, K_x \ket = \int_S \trivial_\La(y) |K(x,y)|^2 d\lambda(y),
\end{align*}
which is the deisred result.
\end{proof}
\subsubsection{Proof of Lemma \ref{lem:Ndelta}}
We start with the following Proposition.
\begin{prop} \label{prop:dif}We have 
\begin{equation}
\|
\rho_{\La} - \bra \cN_{\La} K_{\cdot}, K_{\cdot}\ket\|_{1} 
\le N_{\La} + \E[\Xi(\Lambda)] - 
2 \sum_{j=1}^{N_{\La}} \mu_j^{(\La)}. 
\label{eq:difference1} 
\end{equation}
\end{prop}
\begin{proof}
By Proposition \ref{prop:innerprod1}, we have for all $x \in S_0$
\[
\rho_{\La}(x) - \bra \cN_{\La} K_x, K_x\ket
= \sum_{j=1}^{N_{\La}} (1-\mu_j^{(\La)}) |\Psi_j^{(\La)}(x)|^2
- \sum_{j> N_{\La}} \mu_j^{(\La)}
 |\Psi_j^{(\La)}(x)|^2,
\]
which implies the inequality
\begin{equation}
|\rho_{\La}(x) - \bra \cN_{\La} K_x, K_x\ket |
\leq \sum_{j=1}^{N_{\La}} (1-\mu_j^{(\La)}) |\Psi_j^{(\La)}(x)|^2
+ \sum_{j> N_{\La}} \mu_j^{(\La)}
 |\Psi_j^{(\La)}(x)|^2. \label{ineq:1}
\end{equation}
Using the fact that the function $\Psi_j^{(\La)}$ has $L^2$ norm $1$ for all $j \in \N$, integrating over $x \in S$ in (\ref{ineq:1}) gives
\begin{align*}
\|\rho_{\La} - \bra \cN_{\La} K_\cdot, K_\cdot\ket\|_{1}
&\le \sum_{j=1}^{N_{\La}} (1-\mu_j^{(\La)}) 
+ \sum_{j> N_{\La}} \mu_j^{(\La)} \\
&= N_{\La} - 2 \sum_{j=1}^{N_{\La}} \mu_j^{(\La)} 
+ \sum_{j=1}^{\infty} \mu_j^{(\La)} \\
&= N_{\La} + \E[\Xi(\La)] -2 \sum_{j=1}^{N_{\La}} \mu_j^{(\La)}, 
\end{align*}
where we used (\ref{eq:expectation}) for the last line. The proof is complete.
\end{proof}
For $\delta \in (0,1)$, let $N_\La^\delta$ be the number of eigenvalues of $\cM_\La$ greater than $1-\delta$:
\[
N_{\Lambda}^{(\delta)} := \#\{j \ge 1 : \mu_j^{(\La)} > 1-\delta\}.  
\]
Lemma 3.3 in \cite{AGR16} can be rephrased as follows. Recall that the constant $C_\delta$ has been defined in (\ref{eq:defcdelta}).
\begin{prop}\label{prop:Ndelta}
For $\delta \in (0,1)$, one has 
\begin{align}
 |N_{\La}^{(\delta)} - \E[\Xi(\Lambda)]| 
\le C_\delta \var(\Xi(\La)).
\end{align}
\end{prop}
\begin{proof}
We here give a proof for the reader's convienience though 
the proof is exactly the same as that of Lemma 3.3 in
 \cite{AGR16}. Let $G_{\delta} : [0,1] \to [-1,1]$ be defined
 by 
\[
 G_{\delta}(t) = 
\begin{cases}
 -t & \text{for $t \in [0,1-\delta]$}, \\
 1-t & \text{for $t \in (1-\delta,1]$}.
\end{cases}
\]
It is easy to see that $|G_{\delta}(t)| \le C_{\delta} t(1-t)$ for all $t \in [0,1]$. We also have 
\[
N_{\La}^{(\delta)} - \E[\Xi(\Lambda)]
= \tr(\trivial_{(1-\delta, 1]}(\cM_{\La}) - \cM_{\La} )
= \tr G_{\delta}(\cM_{\La}),
\]
and hence 
\[
|N_{\La}^{(\delta)} - \E[\Xi(\La)]| 
\le C_{\delta} \tr ( \cM_{\La} (I - \cM_{\La})) = \var ( \Xi(\La)),
\]
where we used (\ref{eq:variance0}) for the latter equality. The proof is complete.
\end{proof}
We can now conclude the proof of Lemma \ref{lem:Ndelta}
\begin{proof}[Proof of Lemma \ref{lem:Ndelta}]
We first give a minoration of the sum $\sum_{j=1}^{N_\Lambda} \mu_j^{(\Lambda)}$, using Proposition \ref{prop:Ndelta}. Since the eigenvalues $\mu_j^{(\Lambda)}$ are non-negative and non-increasingly ordered, and by definition of $N_\La^{(\delta)}$, one has 
\begin{equation}\label{ineq:minor}
\sum_{j=1}^{N_\La} \mu_j^{(\La)} \geq \sum_{j=1}^{\min\{ N_\La, N_\La^{(\delta)} \}} \mu_j^{(\La)} \geq (1-\delta)\min\{ N_\La, N_\La^{(\delta)} \}.
\end{equation}
Proposition \ref{prop:Ndelta} implies 
\begin{align}\label{ineq:lemNdelta1}
N_\La^{(\delta)} \geq \E[\Xi(\La)] -C_\delta \var ( \Xi(\La)).
\end{align}
On the other hand, the inequality 
\begin{align}\label{ineq:lemNdelta2}
N_\La \geq \E[\Xi(\La)] - C_\delta \var(\Xi(\La))
\end{align}
easily follows from the definition of $N_\La$ and the positivity of $C_\delta \var(\Xi(\La))$. Inserting inequalities (\ref{ineq:lemNdelta1}) and (\ref{ineq:lemNdelta2}) into inequality (\ref{ineq:minor}), we obtain
\begin{align} \label{ineq:minor2}
\sum_{j=1}^{N_\La} \mu_j^{(\La)} \geq (1-\delta) \{ \E[\Xi(\Lambda)] - C_\delta \var(\Xi(\La))\}.
\end{align}
By Proposition \ref{prop:Ndelta} and inequality (\ref{ineq:minor2}), we have
\begin{align*}
||\rho_\La- \bra \cN_\La K_\cdot , K_\cdot \ket ||_1 &\leq N_\La  + \E[\Xi(\La)]+2(1-\delta) \{C_\delta \var(\Xi(\La))- \E[\Xi(\La)] \} \\
&= N_\La - \E[\Xi(\La)] + 2\delta \E[\Xi(\Lambda)] +2(1-\delta)C_\delta\var(\Xi(\La)) \\
& \leq 1+2\delta \E[\Xi(\Lambda)] +2(1-\delta)C_\delta\var(\Xi(\La)),
\end{align*}
where we used on the last line the fact that $N_\La- \E[\Xi(\La)] \leq 1$, which directly follows from the definition of $N_\La$. The proof is complete.
\end{proof}
\subsubsection{Proof of Lemma \ref{lem:G}}
Lemma \ref{lem:G} is a consequence of Proposition \ref{prop:G}
\begin{proof}[Proof of Lemma \ref{lem:G}] Using the reproducing property of the kernel $K$, we first see that for all $x \in S_0$
\begin{align*}
 K(x,x) \trivial_{\La}(x) = \trivial_{\La}(x) \int_S |K(x,y)|^2 d\lambda(y).
\end{align*}
By Proposition \ref{prop:G}, we now have for all $x \in S_0$ that
\begin{align*}
G(x) 
&= K(x,x) \trivial_{\La}(x) - \bra \cN_{\La} K_x,
 K_x \ket \\
&= 
\trivial_{\La}(x) \int_{S} |K(x,y)|^2
 d\la(y)
- \int_{S} \trivial_{\La}(y) |K(x,y)|^2
 d\la(y) \\
&= 
\int_{S} \{\trivial_{\La}(x) 
-\trivial_{\La}(y) 
\} |K(x,y)|^2
 d\la(y),
\end{align*}
and thus
\begin{align*}
|G(x)| \leq \int_{S} |\trivial_{\La}(x) 
-\trivial_{\La}(y) 
| |K(x,y)|^2
 d\la(y).
\end{align*}
Integrating over $x \in S$ and using formula (\ref{eq:variance}), we obtain
\[
 \|G\|_{1} 
\le \int_S\int_{S} 
|\trivial_{\La}(x) -\trivial_{\La}(y)| 
 |K(x,y)|^2 d\la(x)d\la(y)
= 2 \var(\Xi(\La)),
\]
and the proof is complete.
\end{proof}

\section{Proof of Proposition \ref{prop:hyperunifradial}} \label{sec:radial}
\subsection{Hyperuniformity along balls and proof of Proposition \ref{prop:hyperunifradial}}
Let $d\geq 1$ be an integer and let $\P$ be a radial determinantal point on $\R^d$, in the sense of Definition \ref{def:radial}. It will be seen in the proof of Proposition \ref{prop:hyperunifradial} below that it is sufficient to establish that the point process $\P$ is hyperuniform along balls with linearly increasing radii. We state this fact in the following Proposition we prove further in Section \ref{sec:hyperunifballs}. For $x \in \R^d$ and $r>0$, we denote by $B(x,r)$ the open ball centered at $x$ and of radius $r$.
\begin{prop}\label{prop:varballradial}For any $\kappa >0$, the radial determinantal point process $\P$ is hyperuniform along the exhaustion $\{B(0,R \kappa)\}_{R>0}$.
\end{prop}
\begin{proof}[Proof of Proposition \ref{prop:hyperunifradial}]
Let $\La \subset \R^d$ be open and bounded. By the Vitali covering Theorem, see e.g. \cite{EG92}, there exists a countable family of pairwise disjoint open balls $\{B(a_k,r_k)\}_{k\geq 1} \subset \La$ such that
\begin{align*}
\text{Leb}( \La \setminus \sqcup_{k \geq 1} B(a_k,r_k) ) =0.
\end{align*}
We thus have for all $R>0$ that
\begin{align*}
\var(\Xi(R\La)) = \var (\Xi( \sqcup_{k \geq 1} B(R  a_k,R r_k)).
\end{align*}
By the negative correlation property (\ref{ineq:negcor}), we deduce that
\begin{align} \label{ineq:vitali0}
\var(\Xi(R\La)) \leq \sum_{k \geq 1} \var \left( \Xi\left( B(R a_k,Rr_k)\right)\right) = \sum_{k \geq 1} \var \left( \Xi \left( B(0,R r_k) \right) \right),
\end{align}
where we used the translation invariance of the point process $\P$ for the last equality. Dividing both sides of inequality (\ref{ineq:vitali0}) by $\E[\Xi(\La)]$, we obtain
\begin{align} \label{ineq:vitali1}
\frac{ \var (\Xi( R\La) )}{ \E [\Xi(R\La)]} \leq \sum_{k \geq 1} \frac{ \var  \left( \Xi\left( B(0,Rr_k)\right)\right) }{\E [ \Xi(R \La)]}.
\end{align}
We prove that each term in the sum on the right hand side of (\ref{ineq:vitali1}) goes to 0 as $R$ goes to infinity. For any $k\geq 1$ and any $R>0$, we have $\text{Leb}(B(0,R r_k)) \leq \text{Leb}(R\La)$ and thus
\begin{align} \label{ineq:vitali2}
\frac{\var  \left( \Xi\left( B(0,Rr_k)\right)\right) }{\E [ \Xi(R \La)]} \leq \frac{\var  \left( \Xi\left( B(0,Rr_k)\right)\right) }{\E [ \Xi(B(0,Rr_k))]}.
\end{align}
By Proposition \ref{prop:varballradial}, the right hand side of (\ref{ineq:vitali2}) goes to 0 as $R$ goes to infinity. In order to apply the dominated convergence Theorem, observe that we have
\begin{align*}
\frac{\var  \left( \Xi\left( B(0,Rr_k)\right)\right) }{\E [ \Xi(R \La)]} \leq \frac{\E[\Xi( B(0,Rr_k))]}{\E[\Xi(R\La)]} = \frac{\text{Leb}(B(0,Rr_k))}{\text{Leb}(R\La)}= \frac{\text{Leb}(B(0,r_k))}{\text{Leb}(\La)},
\end{align*}
and that
\begin{align*}
\sum_{k \geq 1} \frac{\text{Leb}(B(0,r_k))}{\text{Leb}(\La)} = 1.
\end{align*}
The proof is complete.
\end{proof}
\subsection{Proof of Proposition \ref{prop:varballradial}} \label{sec:hyperunifballs}
Proposition \ref{prop:varballradial} will be a consequence of a geometric interpretation of the number variance given in terms of volumes of intersections of balls. This method is adapted from \cite[Prop. 1]{demnilazag}, concerning a family of determinant point processes in the Poincar\'e disk invariant under the action of M\"obius transformations. We denote by $B(0,R)^c$ the complementary set of the ball $B(0,R)$.
\begin{prop} \label{prop:varlunule} Let $\P$ be a radial point process on $\R^d$ with correlation kernel $K$. Let $\varphi : [0,+ \infty) \to [0,+\infty)$ be the function such that
\begin{align*}
|K(x,y)|^2= \varphi(||x-y||)
\end{align*}
for all $x,y \in \R^d$. Then we have for any $R>0$ the exact formula
\begin{align}
\var ( \Xi(B(0,R)) = \int_{\R^d} \varphi(||x||)\mathrm{Leb}(B(0,R)^c \cap B(x,R)) dx.
\end{align}
\end{prop}
\begin{proof}
From Formula (\ref{eq:variance2}), we have
\begin{align*}
\var( \Xi(B(0,R)) = \int_{y \in B(0,R)^c} \int_{x \in B(0,R)} \varphi(||x-y||) dx dy.
\end{align*}
Performing the change of variable $ x\mapsto x+y$, we obtain
\begin{align*}
\var( \Xi(B(0,R)) = \int_{y \in B(0,R)^c} \int_{x \in B(y,R)} \varphi(||x||)dxdy.
\end{align*}
By Fubini's Theorem and since we have
\begin{align*}
y \in \R^d \setminus B(0,R) , \hspace{0.1cm} x \in B(y,R) \hspace{0.1cm} \Leftrightarrow \hspace{0.1cm} x \in \R^d \setminus \{0\}, \hspace{0.1cm} y \in B(0.R)^c \cap B(x,R),
\end{align*}
we obtain
\begin{align*}
\var ( \Xi(B(0,R)) = \int_{\R^d} \varphi(||x||)\mathrm{Leb}(B(0,R)^c \cap B(x,R)) dx,
\end{align*}
and the proof is complete.
\end{proof}
The intersection volume can be explicitly computed. For a positive integer $d$, we denote by $c_d$ the volume of the unit ball $B(0,1) \subset \R^d$,
\[c_{d} = \frac{\pi^{d/2}}{\Gamma(d/2+1)}. \]
\begin{prop}\label{prop:intersecvolume}
For all $R>0$ and all $x \in B(0,2R)$, we have
\begin{equation}\label{eq:area-lunule}
\mathrm{Leb}(B(0,R)^c \cap B(x,R))=2c_{d-1}R^d . \sum_{k=0}^{+\infty}\frac{\left(-\frac{d-1}{2}\right)_k}{(2k+1)k!}\left( \frac{||x||}{2R}\right)^{2k+1},
\end{equation}
where we use the Pochhammer symbol
\begin{align*}
(\alpha)_k=\alpha(\alpha+1)...(\alpha+k-1).
\end{align*}
\end{prop}
\begin{proof}
By symmetry, we can assume that $x = r(1,0,...,0)^t$, where $0\leq r<2R$. We obviously have
\begin{align*}
\mathrm{Leb}(B(0,R)^c \cap B(x,R))= \text{Leb}(B(0,R))- \text{Leb}(B(x,R) \cap B(0,R)).
\end{align*}
Besides, we have
\begin{align*}
y \in B(0,R) \cap B(x,R) \Leftrightarrow x/2-y \in B(0,R) \cap B(x,R),
\end{align*}
and also
\begin{align*}
y=(y_1,...,y_d)^t \in B(0,R), \hspace{0.1cm} y_1 \geq r/2 \Rightarrow y \in B(x,R).
\end{align*}
This implies
\begin{equation}\label{eq:intersec}
\mathrm{Leb}( B(0,R) \cap B(x,R) )=2\text{Leb}(y=(y_1,...,y_d)^t \in B(0,R), \hspace{0.1cm}y_1 \geq r/2).
\end{equation}
For $s \in [r/2,R]$, we denote by $H_s$ the affine hyperplane $\{ y = (y_1,...,y_d)^t \in \R^d, \hspace{0.1cm}y_1=s \}$. The intersection $H_s \cap  B(0,R) \cap B(x,R)= H_s \cap B(0,R)$ is a $(d-1)$-ball of radius $\sqrt{R^2-s^2}$ and with $(d-1)$-Lebesgue measure
\begin{align*}
\frac{\pi^{\frac{d-1}{2}}\sqrt{R^2-s^2}^{d-1}}{\Gamma\left( \frac{d-1}{2}+1\right)} = c_{d-1}\sqrt{R^2-s^2}^{d-1}.
\end{align*}
We obtain from equation (\ref{eq:intersec}) that
\begin{align*}
\mathrm{Leb}( B(0,R) \cap B(x,R) )&=2 c_{d-1}\int_{r/2}^R\sqrt{R^2-s^2}^{d-1}ds =2c_{d-1} R^{d-1}\int_{r/2}^R\sqrt{1- (s/R)^2}^{d-1}ds
\end{align*}
Performing the change of variable $u= s/R$, we have
\begin{align*}
\mathrm{Leb}( B(0,R) \cap B(x,R) )&=2c_{d-1} R^{d} \int_{r/2R}^1 \left(1-u^2\right)^{\frac{d-1}{2}}du \\
&=2c_{d-1} R^{d} \int_{r/2R}^1 \left( \sum_{k=0}^{+ \infty} \frac{\left(-\frac{d-1}{2} \right)_k}{k!}u^{2k} \right) du\\
&=2c_{d-1} R^{d} \left\lbrace \sum_{k=0}^{+\infty}\frac{\left(-\frac{d-1}{2}\right)_k}{(2k+1)k!} \left(1-\left(\frac{r}{2R}\right)^{2k+1}\right) \right\rbrace.
\end{align*}
The intervertion sum integral performed on the last line is justified since the series
\[\sum_{k=0}^{+ \infty} \frac{\left(-\frac{d-1}{2} \right)_k}{k!}
\]
is convergent, either as a finite sum when $d$ is odd, while for even $d$, one has by Stirling formula
\begin{align*}
\frac{\left( -\frac{d-1}{2}\right)_k}{k!} =\frac{\Gamma\left( k- \frac{d-1}{2}\right)}{\Gamma\left( -\frac{d-1}{2}\right)\Gamma(k+1)} \underset{k\rightarrow +\infty}{\sim} \frac{k^{-(d-1)/2-1}}{\Gamma\left( -\frac{d-1}{2}\right)}.
\end{align*}
With $r=0$, we obtain in particular the expression
\begin{align*}
\text{Leb}(B(0,R))=2c_{d-1} R^{d} \sum_{k=0}^{+\infty}\frac{\left(-\frac{d-1}{2}\right)_k}{(2k+1)k!},
\end{align*}
from which we deduce
\begin{align*}
\text{Leb}(B(0,R)^c \cap B(x,R))&=2c_{d-1} R^{d}\sum_{k=0}^{+\infty}\frac{\left(-\frac{d-1}{2}\right)_k}{(2k+1)k!}\left( \frac{r}{2R}\right)^{2k+1},\\
\end{align*}
and the proof is complete.
\end{proof}
We can now conclude the proof of Proposition \ref{prop:varballradial}. For simplicity, we take $\kappa=1$ and prove that the radial point process $\P$ is hyperuniform along the exhaustion $\{B(0,R)\}_{R>0}$, the proof for the exhaustion $\{B(0,R \kappa )\}_{R >0}$ being similar. We first use Proposition \ref{prop:varlunule} and write, for $R>0$
\begin{multline} \label{eq:varball1}
\frac{\var ( \Xi(B(0,R))}{\E[\Xi(B(0,R)]}= \frac{1}{\mathrm{Leb}(B(0,R))} \int_{B(0,2R)} \varphi( ||x|| ) \mathrm{Leb}(B(0,R)^c \cap B(x,R) ) dx \\
+ \int_{B(0,2R)^c} \varphi(||x||) dx.
\end{multline}
By the reproducing property (\ref{eq:repprop}), the integral
\begin{align} \label{eq:reppropint}
\int_{\R^d} \varphi(||x||)dx= \int_{\R^d} |K(x,0)|^2dx= K(0,0)=1
\end{align}
is finite, and thus
\begin{align*}
\int_{B(0,2R)^c} \varphi(||x||) dx \to 0
\end{align*}
as $R \to + \infty$. Using polar coordinates and Proposition (\ref{prop:intersecvolume}), we rewrite equation (\ref{eq:varball1}) as
\begin{align} \label{eq:hyperunifball}
\frac{\var ( \Xi(B(0,R)) )}{\E[\Xi(B(0,R)]} = 2 d c_{d-1}\sum_{k=0}^{+\infty} \frac{ \left(-\frac{d-1}{2} \right)_k}{(2k + 1) k!} \frac{1}{(2R)^{2k +1}} \int_0^{2R} r^{2k +1} r^{d-1}\varphi(r)dr + o(1).
\end{align}
The constant factor $2 d c_{d-1}$ above is obtained from the expression for the $(d-1)$-Lebesgue measure of the unit sphere $\sigma_{d-1}$ given by
\begin{align*}
\sigma_{d-1} = d c_d.
\end{align*}
We now use the following Lemma in order to treat the integrals
\begin{align*}
\frac{1}{(2R)^{2k +1}} \int_0^{2R} r^{2k +1} r^{d-1}\varphi(r)dr.
\end{align*}
\begin{lem} \label{lem:dense} Let $g \in L^1([0, + \infty), dr)$ and let $\eta >0$. Then we have
\begin{align}
\lim_{R \to + \infty} \frac{1}{(2R)^{\eta}} \int_0^{2R} r^\eta g(r) dr =  0.
\end{align}
\end{lem}
\begin{proof}
We use the fact that bounded compactly supported functions are dense in $L^1([0,+\infty),dr)$. Let $\varepsilon>0$ and let $g_0 \in L^1([0, + \infty),dr)$ be bounded and compactly supported such that
\begin{align} \label{ineq:densel1}
\int_0^{+\infty} |g(r)-g_0(r)|dr \leq \varepsilon.
\end{align}
We now write
\begin{align} \label{ineq:triangl1}
\frac{1}{(2R)^\eta} \left| \int_0^{2R} r^\eta g(r) dr \right| \leq \frac{1}{(2R)^\eta} \int_0^{2R} r^\eta |g(r)-g_0(r)| dr + \frac{1}{(2R)^\eta} \int_0^{2R} r^\eta |g_0(r)|dr.
\end{align}
Since the function $g_0$ is bounded and compactly supported, we have
\begin{align*}
\frac{1}{(2R)^\eta} \int_0^{2R} r^\eta |g_0(r)|dr \to 0
\end{align*}
as $R \to + \infty$. Since
\begin{align*}
\frac{1}{(2R)^\eta} \int_0^{2R} r^\eta |g(r)-g_0(r)| dr \leq \int_0^{2R} |g(r)-g_0(r)|dr
\end{align*}
and recalling inequalities (\ref{ineq:densel1}) and (\ref{ineq:triangl1}), we obtain
\begin{align*}
\limsup_{R \to + \infty} \frac{1}{(2R)^\eta} \left| \int_0^{2R} r^\eta g(r) dr \right| \leq \varepsilon,
\end{align*}
and the proof is complete.
\end{proof}
We apply Lemma \ref{lem:dense} to the function $g(r)=r^{d-1}\varphi(r)$, which is integrable by formula  (\ref{eq:reppropint}), and the exponents $\eta = 2k+1$ to obtain for any $k\geq 0$
\begin{align*}
\lim_{R \to + \infty} \frac{1}{(2R)^{2k +1}}  \int_{0}^{2R} r^{2k +1} r^{d-1} \varphi(r) dr =0.
\end{align*}
Recalling equation (\ref{eq:hyperunifball}), we conclude the proof of Proposition \ref{prop:varballradial} by applying the dominated convergence Theorem.
\begin{rem}
    One could ask, given a function $\varphi : [0, + \infty ) \to [0, + \infty)$, whether there exists a kernel $K$ on $\R^d \times \R^d$ of an orthogonal projection operator such that $|K(x,y)|^2 = \varphi ( || x-y ||)$. A necessary condition is obtained from the reproducing property (\ref{eq:repprop}) for the kernel $K$ that implies
    \begin{align*}
        \int_{\R^d} \varphi(||x-y||)dy = \int_{\R^d} \varphi(||x||) dx = dc_d \int_0^{+ \infty} r^{d-1} \varphi(r) dr =  \varphi(0)^{1/2}.
    \end{align*}
\end{rem}
\section{Examples} \label{sec:ex}
\subsection{Partial isometries}
\subsubsection{A general framework}
Partial isometries in the context of determinantal point processes have been studied in \cite{KS22a}. We consider another locally compact complete and separable metric space $S'$ equiped with a Radon measure $\lambda'$.
\begin{defi}A \emph{partial isometry} between $L^2(S,\lambda)$ and $L^2(S',\lambda')$ is a linear map
\begin{align*}
\cW : L^2(S,\lambda) \to L^2(S', \lambda')
\end{align*}
such that
\begin{align*}
||\cW f ||_{L^2(S',\lambda')} = ||f||_{L^2(S,\lambda)}
\end{align*}
for all $f$ in the orthogonal complement of $\ker \cW$. 
\end{defi}
If the linear map $\cW :  L^2(S,\lambda) \to L^2(S', \lambda') $ is a locally Hilbert-Schmidt partial isometry, then the operator $\cK:=\cW^* \cW$ on $L^2(S,\lambda)$ is the locally trace class orthogonal projection onto the closure of the range of $\cW^*$, which is the orthogonal complement of the kernel of $\cW$.

For a relatively compact subset $\La \subset S$, we define the self-adjoint operator $\cH_\La = \cW \cP_{\La} \cW^*$ and consider its spectral decomposition
\[ \cH_\La:= \sum_{j=1}^{+ \infty} \mu_j^{(\La)} h_j^{(\La)} \otimes h_j^{(\La)}. \]
Recall that the functions $\Psi_j^{(\La)}$ together with the accumulated spectrogram $\rho_\La$ associated to $(\cK,\La)$ have been defined through the polar decomposition of $\cK \cP_\La$. In such a context, wa have another interpretation of the accumulated spectrogram.
\begin{prop} \label{prop:partialisometries} We have
\begin{align}
    \rho_\La(x) = \sum_{j=1}^{N_\La} | \cW^* h_j^{(\La)} (x) |^2.
\end{align}
\end{prop}
\begin{proof}
    We first notice that the spectral decomposition of the trace class hermitian operator $\cK \cP_\La \cK$ can be written as
    \begin{align}
        \cK \cP_\La \cK = \sum_{j=1}^{+\infty} \mu_j^{(\La)} \Psi_j^{(\La)} \otimes \Psi_j^{(\La)}.
    \end{align}
    Indeed, the operators $\cK \cP_\La \cK$ and $\cN_\La = \cK \cP_\La$ have the same non-zero eigenvalues, and  for all $j \in \N$ such that $\mu_j^{\La } \neq 0$, we have from (\ref{eq:exprpsi}) that
    \begin{align*}
        \cK \cP_\La \cK \Psi_j^{(\La)} &= \frac{1}{\sqrt{\mu_j^{(\La)}}} \cK \cP_\La \cK (\cK \cP_\La )\Phi_j^{(\La)} 
        = \frac{1}{\sqrt{\mu_j^{(\La)}}} \cK \cP_\La ( \cP_\La \cK \cP_\La) \Phi_j^{(\La)}  \\
        &= \sqrt{\mu_j^{(\La)}} \cK \cP_\La \Phi_j^{(\La)}
        = \mu_j^{(\La)} \Psi_j^{(\La)}.
    \end{align*}
    The claim now follows from the fact that we have the conjugation
    \begin{align}
        \cK \cP_\La \cK = \cW^* ( \cW \cP_{\La} \cW^* ) \cW.
    \end{align}
\end{proof}
\subsubsection{Short-time Fourier transform and space-time localization operators}
Within the framework described above, in \cite{AGR16}, the authors consider the partial isometry $\cW$ to be the adjoint of the short-time Fourier transform $V_g$ defined as follows (see also \cite{Gro01} for further background): for a {\it window} function $g \in L^2(\R^d,dx)$ with $||g||_2 =1$, the operator $V_g : L^2(\R^d) \to L^2(\R^{2d})$ is defined by
\begin{align*}
    V_g f (x, \xi) = \int_{\R^d} f(t) g(t-x) e^{- 2i \pi t \cdot \xi} dt, \quad f \in L^2(\R^d), \hspace{0.1cm} (x,\xi) \in \R^{2d}.
\end{align*}
Its adjoint operator is given by
\begin{align*}
    V_g^* F (t) = \int_{\R^{2d}} F(x,\xi) g(t-x)e^{2 i \pi \xi \cdot t} dx d\xi, \quad F \in L^2(\R^{2d}), \hspace{0.1cm} t \in \R^d. 
\end{align*}
For a bounded set $\La \subset \R^{2d}$, the time-frequency localization operator is then defined by:
\begin{align}
    \cH_\La := V_g^* \cP_\La V_g,
\end{align}
and has the spectral decomposition
\begin{align*}
    \cH_\La = \sum_{j \geq 1} \mu_j^{(\La)} h_j^{(\La)} \otimes h_j^{(\La)}.
\end{align*}
The authors in \cite{AGR16} then define the accumulated spectrogram to be the function
\begin{align} \label{def:accspecAGR}
    \rho_\La (x) = \sum_{j=1}^{N_\La} |V_g h_j^{(\La)}(x)|^2.
\end{align}
From Proposition \ref{prop:partialisometries} above, we see that the definition \ref{def:accspec} we use coincides with (\ref{def:accspecAGR}). In particular, our main results Theorem \ref{thm:L1-limit} and Corollaries \ref{cor:L1-lim1} and \ref{cor-L1lim2} are generalizations of Theorem 1.3 in \cite{AGR16}.
\subsubsection{The Heisenberg family of determinantal point processes}
The state space is now $\C^d = \R^{2d}$, equipped with the euclidean norm $|z|^2=|z_1|^2 + \dots +|z_d|^2$ for $z=(z_1,\dots,z_d) \in \C^d$ and the Lebesgue measure $dm(z)$. Taking $g$ to be the normalized gaussian function
\begin{align*}
    g(x)= 2^{d/4} \exp \left( - \pi ||x||^2 \right),
\end{align*}
and writing  $z = x + i\xi$, we obtain that the kernel $K$ of the operator $\cK = V_g V_g^*$ defined in the preceding section is given by
\begin{align*}
K(z,w)= e^{\pi(z \cdot \overline{w} - \frac{|z|^2}{2} - \frac{|w|^2}{2})},
\end{align*}
so that
\begin{align*}
|K(z,w)|^2 = e^{-\pi|z-w|^2}.
\end{align*}
Up to the change of the measure $dm(z) \mapsto e^{-\pi|z|^2}dm(z)$, the kernel $K$ is the reproducing kernel of the Bargmann-Fock space. We let $\P$ be the radial determinantal point process with kernel $K$. In the case $d=1$, the point process $\P$ is nothing but the infinite Ginibre enemble (\cite{Gin65}). Several generalizations of these point processes have been considered, such as determinantal point processes corresponding to poly-analytic Fock spaces or higher Landau levels (\cite{Shirai15}, \cite{APRT17}, \cite{K22}, \cite{MKS21}), as well as finite versions (\cite{AGR19}, \cite{HaimiHedenmalm13}), establishing hyperuniformity along balls. The name \emph{Heisenberg} is justified from relations with the representation theory of the Heisenberg group. Note also that \cite{K22} and \cite{MKS21} give a full asymptotic expansion for the number variance of particles inside a ball of large radius.
\subsection{The Euclidean family of determinantal point processes associated to Paley-Wiener spaces}
We derscribe here what we call the \emph{Euclidean family of determinantal point processes}, according to the terminology of \cite{KS22a}, see also \cite{Zel00}. The term \emph{Euclidean} refers to the invariance by the group of Euclidean motions of this family of determinantal point processes, i.e. they are both rotationally invariant and translation invariant. These radial determinantal point processes were also studied in \cite{SZT2009}, \cite{TSZ2008} under the name \emph{Fermi-sphere ensembles}. They also appear in the context of semi-classical analysis in \cite{DeleporteLambert}, \cite{KS22b}, see also \cite{Zel00}. Below, we use the method of the proof of Proposition \ref{prop:varballradial} to that example to derive the first order asymptotics for the number variance of particles inside a ball of large radius.

We now give a proper defintion. We let $\cF$ be the Fourier transform on $L^2(\R^d,dx)$
\begin{align*}
\cF f(\xi) = \frac{1}{(2\pi)^d}\int_{\R^d} f(x) e^{-i \xi \cdot x} dx. 
\end{align*}
The Paley-Wiener space is defined as the subspace of $L^2(\R^d,dx)$ of functions having their Fourier transform supported in the unit ball $B(0,1)$. The projection operator onto the Paley-Wiener space is the operator $\cF^{-1} \cP_{B(0,1)} \cF$ and its kernel reads
\begin{align*}
\mathcal{S}_d(x,y) = \frac{1}{(2\pi)^d}\int_{B(0,1)} e^{i \xi  \cdot(x- y)} d\xi = \frac{1}{(2\pi)^{d/2}} \frac{J_{d/2}(||x-y||)}{||x-y||^{d/2}},
\end{align*}
where $J_{d/2}$ is the Bessel function of order $d/2$ of the first kind
\begin{align*}
J_{d/2}(x)= \sum_{k=0}^{+ \infty} \frac{(-1)^k}{k!\Gamma(1+d/2+k)} \left(\frac{x}{2}\right)^{2k + d/2}.
\end{align*}
\begin{rem}
    If, instead of considering the unit ball $B(0,1)$, one considers a set $A \subset \R^d$ that is bounded and rotationally symmetric, i.e. if $OA = A$ for any orthogonal transformation $O \in O(d)$, then the operator $\cK_A := \cF^{-1} \cP_A \cF $ is a locally trace class orthogonal projection on $L^2(\R^d, dx)$ with kernel
    \begin{align*}
        K_A (x,y) \frac{1}{(2\pi)^d} \int_A e^{i \xi \cdot (x-y)} d\xi,
    \end{align*}
    that is also radially symmetric. Thus, the operator $\cK_A$ gives rise to another example of a radial determinantal point process on $\R^d$.
\end{rem}
\begin{defi}The \emph{Euclidean family of determinantal point processes} $\P_{\cS_d}$ on $\R^d$, $d=1,2,\dots$, is formed by the radial determinantal point processes on $\R^d$ with correlation kernel $\mathcal{S}_d$.
\end{defi}
For $d=1$, one recovers the usual Paley-Wiener space and the Dyson's sine process with kernel
\begin{align*} 
\mathcal{S}_1(x,y)= \frac{\sin(x-y)}{\pi(x-y)}.
\end{align*}
The hyperuniformity of the Euclidean family of determinantal point processes $\P_{\cS_d}$ follows from the genral Propositions \ref{prop:hyperunifradial} and \ref{prop:varballradial}. Using known asymptotics of the Bessel functions allows to obtain a more precise statement than Proposition \ref{prop:varballradial} given by the following Proposition.
\begin{prop}\label{prop:varsinc} As $R \to + \infty$, we have the asymptotics
\begin{align*}
\var(\Xi(B(0,R))= \frac{1}{2^{d-1}\pi^{3/2}\Gamma \left( \frac{d+1}{2}\right) \Gamma \left( \frac{d}{2} \right)} \log(R)R^{d-1} + O(R^{d-1}).
\end{align*}
\end{prop}
\begin{proof}
Set
\begin{align*}
g(r)= \frac{J_{d/2}(r)^2}{r}, \quad r \geq 0,
\end{align*}
so that Equation (\ref{eq:varball1}) reads in that particular case
\begin{align} \label{eq:varballsinc}
\var(\Xi(B(0,R)) &=  \frac{2c_{d-1} \sigma_{d-1}}{(2\pi)^d} \sum_{k=0}^{+ \infty} \frac{\left(- \frac{d-1}{2}\right)_k}{k!(2k+1)} \frac{R^d}{(2R)^{2k+1}}\int_0^{2R} r^{2k+1}g(r)dr \\
&+ c_d \sigma_{d-1}R^d \int_{2R}^{+ \infty} g(r) dr. \nonumber
\end{align}
As $r \to + \infty$, we have the following asymptotics for the Bessel function
\begin{align*}
J_{d/2}(r)^2 \sim \frac{1+ \cos\left(2r -\frac{(d-1)\pi}{2}\right)}{\pi r},
\end{align*}
from which we deduce that
\begin{align} \label{equiv1}
\int_0^{2R} r g(r)dr = \frac{1}{\pi }\log(R) + O(1),
\end{align}
as well as
\begin{align}\label{equiv2}
\int_0^{2R}r^{2k+1}g(r) dr = O(R^{2k})
\end{align}
uniformly for $k \geq 1$, and finally
\begin{align} \label{equiv3}
\int_{2R}^{+\infty} g(r)dr = O(R^{-1}).
\end{align}
Inserting the estimates (\ref{equiv1}), (\ref{equiv2}) and (\ref{equiv3}) into Equation (\ref{eq:varballsinc}), we obtain
\begin{align*}
\var( \Xi(B(0,R)) = \frac{c_{d-1} \sigma_{d-1}}{2^{d-1} \pi^{d+1}} R^{d-1} \log(R) + O(R^{d-1}).
\end{align*}
The constant factor is equal to
\begin{align*}
\frac{c_{d-1} \sigma_{d-1}}{2^{d-1} \pi^{d+1}} = \frac{1}{2^{d-1}\pi^{3/2}\Gamma \left( \frac{d+1}{2}\right) \Gamma \left( \frac{d}{2} \right)},
\end{align*}
which concludes the proof.
\end{proof}

%%%%%%%% Figures %%%%%%%%%
%\begin{figure}[htbp]
%\begin{center}
%\includegraphics{triangle.eps}
%\includegraphics[width=1.0\hsize]{planecurve.pdf}
%\end{center}
%\caption{Schwarz-Christoffel map}
%\end{figure}

%%%%%%%% References %%%%%%%%%

%\bibliography{foo.bib} 

\end{document}